\documentclass[12pt]{amsart}
\pdfoutput=1
\usepackage{amsmath, amssymb}
\usepackage{amsfonts}
\usepackage{mathrsfs}
\usepackage[arrow,matrix,curve,cmtip,ps]{xy}
\usepackage{graphicx}
\usepackage{amsthm}
\usepackage{float}
\usepackage[margin=1.4 in]{geometry}
\usepackage{amsthm}
\usepackage[utf8]{inputenc}
\usepackage[T1]{fontenc}
\usepackage{mathtools}

\allowdisplaybreaks

\newtheorem{theorem}{Theorem}[section]
\newtheorem{lemma}[theorem]{Lemma}

\newtheorem{corollary}[theorem]{Corollary}

\newtheorem*{theorem*}{Theorem}
\theoremstyle{remark}
\newtheorem{remark}[theorem]{Remark}

\newtheorem{question}{Question}

\numberwithin{equation}{section}

\begin{document}
\title[Bounds on $k$-systems via the mapping class group]{  local geometry of the $k$-curve graph }

\author{Tarik Aougab}

\address{Department of Mathematics \\ Brown University \\ 151 Thayer Street, Providence, RI 02902 \\ USA}

\date{\today}


\begin{abstract}

Let $S$ be an orientable surface with negative Euler characteristic. For $k \in \mathbb{N}$, let $\mathcal{C}_{k}(S)$ denote the \textit{k-curve graph}, whose vertices are isotopy classes of essential simple closed curves on $S$, and whose edges correspond to pairs of curves that can be realized to intersect at most $k$ times. The theme of this paper is that the geometry of Teichm{\"u}ller space and of the mapping class group captures local combinatorial properties of $\mathcal{C}_{k}(S)$, for large $k$. Using techniques for measuring distance in Teichm{\"u}ller space, we obtain upper bounds on the following three quantities for large $k$: the clique number of $\mathcal{C}_{k}(S)$ (exponential in $k$, which improves on previous bounds of \cite{JuMalMo} and \cite{Prz} and which is essentially sharp); the maximum size of the intersection, whenever it is finite, of a pair of links in $\mathcal{C}_{k}$ (quasi-polynomial in $k$); and the diameter in $\mathcal{C}_{0}(S)$ of a large clique of $\mathcal{C}_{k}(S)$ (uniformly bounded). As an application, we obtain quasi-polynomial upper bounds, depending only on the topology of $S$, on the number of short simple closed geodesics on any square-tiled surface homeomorphic to $S$. 
\end{abstract}

\maketitle

\section{Introduction} Let $S$ be an orientable surface with negative Euler characteristic. The \textit{curve graph} of $S$, denoted $\mathcal{C}(S)$, is the graph whose vertices correspond to isotopy classes of essential simple closed curves on $S$, and such that there is an edge between isotopy classes that can be realized disjointly on $S$. The curve graph has deep connections to the geometry of Teichm{\"u}ller space $\mathcal{T}(S)$ and to the mapping class group $\mbox{Mod}(S)$. Indeed, as a metric space it is quasi-isometric to the \textit{electrified Teichm{\"u}ller space}-- the space obtained from $\mathcal{T}(S)$ equipped with the Teichm{\"u}ller metric by coning off, for each simple closed curve $\alpha$, the region associated to those hyperbolic surfaces on which $\alpha$ is very short (\cite{MasMin1}). Moreover, the group of simplicial automorphisms of $\mathcal{C}(S)$ is isomorphic to $\mbox{Mod}^{\pm}(S)$ (\cite{Ivanov}, \cite{Kork}, \cite{Luo}), the extended mapping class group. 

 In this paper, we consider for each $k \in \mathbb{N}$, a variant of $\mathcal{C}(S)$ called the \textit{k-curve graph} and denoted $\mathcal{C}_{k}(S)$: vertices are the same as $\mathcal{C}(S)$, and edges correspond to pairs of isotopy classes that can be realized with at most $k$ intersections. The large scale geometry of $\mathcal{C}_{k}(S)$ is well-understood, because it is quasi-isometric to the standard curve graph $\mathcal{C}(S)=\mathcal{C}_{0}(S)$. However, the local combinatorics of $\mathcal{C}_{k}(S)$, and how they depend on $k$, remain largely unexplored. The theme of this paper is that large scale geometric features of $\mathcal{T}(S)$ can be translated into local geometric features of $\mathcal{C}_{k}(S)$, when $k$ is large.

\subsection{Cliques in $\mathcal{C}_{k}(S)$.}

As motivation, we recall the following question, first popularized by Farb and Leininger:

\begin{question} \label{FL} As a function of $S$, what is the largest size of a collection $\Omega$ of pairwise non-homotopic, essential simple closed curves, such that no two curves in $\Omega$ intersect more than once?
\end{question}

Question \ref{FL} is surprisingly challenging and remains open, although progress has been made towards its resolution. Most recently, Przytycki has shown that any such $\Omega$ has size bounded above by an explicit function that grows as a cubic polynomial in $|\chi(S)|$ (\cite{Prz}). On the other hand, it is not difficult to construct sequences of such collections whose cardinalities grow quadratically in $|\chi(S)|$ (\cite{Aoug1}, \cite{MalRivTher}). A natural generalization of Question \ref{FL} is to ask, as a function of $k \in \mathbb{N}$, for the largest size of a collection of (pairwise non-homotopic) simple closed curves, pairwise intersecting at most $k$ times. We call such a collection of curves a \textit{k-system}.

Indeed, Przyticky's bounds apply to this generalization, and in particular his result states that the maximum size of a $k$-system grows at most as a polynomial in $|\chi(S)|$ of degree $k^{2}+k+1$. Juvan-Malni\v{c}-Mohar have also considered this question, and when $k$ is large compared to $|\chi(S)|$, they show that such a collection has size roughly at most $k^{k}$ (\cite{JuMalMo}). 

Question \ref{FL} can be reinterpreted as asking for the largest size of a clique in $\mathcal{C}_{1}(S)$, and thus its generalization asks for the largest clique size in $\mathcal{C}_{k}(S)$. An \textit{n-clique} is a complete graph on $n$ vertices, and the \textit{clique number} of a graph $G$ is the supremum over all $n$ such that there exists an embedded $n$-clique in $G$. Our first result provides an upper bound for the clique number of $\mathcal{C}_{k}$, which when $k$ is large, outperforms the bounds from \cite{Prz} and from \cite{JuMalMo} (see subsection $1.4$ below for notation): \vspace{3 mm}

\textbf{Theorem \ref{KSYS}.}
\textit{Fix a surface $S$ with $\chi(S)<0$, and let $N_{S}(k)$ denote the clique number of $\mathcal{C}_{k}(S)$.  Then } 

\[ \log(N_{S}(k)) \prec k.  \]

Thus, for a fixed surface $S$, $N_{S}(k)$ grows at most exponentially as a function of $k$. In \cite{Aoug1} we showed, for each $g$, the existence of a complete subgraph of $\mathcal{C}_{k}(S_{g})$ whose size was on the order of $g^{k/2}$, where $S_{g}$ is the closed surface of genus $g$. It follows that Theorem \ref{KSYS} is essentially sharp.

\subsection{Intersections of links in $\mathcal{C}_{k}(S)$.}

Given a simple closed curve $\alpha$ on $S$, then the $k$-\textit{link} of $\alpha$, denoted $\mathcal{L}_{k}(\alpha)$, is the sphere of radius $1$ in $\mathcal{C}_{k}(S)$, centered at the vertex associated to $\alpha$. If $\alpha$ and $\beta$ \textit{fill} $S$, that is, if $S \setminus (\alpha \cup \beta)$ is a disjoint union of topological disks, boundary parallel annuli and once-punctured disks, then $|\mathcal{L}_{k}(\alpha) \cap \mathcal{L}_{k}(\beta)|$ is finite, for all $k$. On the other hand if $\alpha$ and $\beta$ do not fill, this intersection can easily be infinite. Our next result states that for fixed $k$, $|\mathcal{L}_{k}(\alpha) \cap \mathcal{L}_{k}(\beta)|$ is \textit{uniformly} bounded over all choices of filling pairs $\alpha, \beta$, and furthermore that this bound grows at most quasi-polynomially in $k$. A function $f: \mathbb{N} \rightarrow \mathbb{N}$ grows \textit{at most quasi-polynomially} if there exists some positive $c, \lambda \geq 1$ so that 
\[ f(n) \leq 2^{(\lambda \cdot \log(n))^{c}}. \]
When $c= 1$ the left hand side is bounded above by a polynomial of degree $\lceil \lambda \rceil$. 
\vspace{2 mm}

\textbf{Theorem \ref{KLink}.} 
\textit{There exists a function $r_{S}(k)$ depending only on the topology of $S$, which grows at most quasi-polynomially and which satisfies the following. Let $\alpha, \beta$ be simple closed curves on $S$ which fill $S$, and let $\mathcal{L}_{k}(\alpha)$ denote the set of all vertices in $\mathcal{C}_{k}(S)$ that are distance $1$ from $\alpha$. Then $|\mathcal{L}_{k}(\alpha) \cap \mathcal{L}_{k}(\beta)|\leq r_{S}(k)$.  }

We remark that our methods also prove Theorem \ref{KLink} when $\alpha$ and $\beta$ are allowed to be multi-curves: there exists a uniform upper bound on the numer of simple closed curves intersecting a filling multi-curve pair at most $k$ times, and this bound grows at most sub-exponentially in $k$.

In \cite{Aoug1} we asked whether large $k$-systems project to small diameter subsets of the curve graph $\mathcal{C}_{0}(S)$. In particular, if $i(\alpha, \beta) \leq k$, $\alpha$ and $\beta$ are at most roughly $\log(k)$ apart in the curve graph, and we asked if there exists large $k$-systems obtaining this theoretical upper bound on diameter in $\mathcal{C}_{0}(S)$. Theorem \ref{KLink} answers this in the negative; indeed, any large $k$-system projects to a diameter $2$ subset of $\mathcal{C}_{0}(S)$:

\textbf{Corollary \ref{CGDiam}.}
\textit{Let $\Omega$ be a $k$-system on $S$ with $|\Omega|= N_{k}(S)$. Then for all sufficiently large $k$, $\Omega$ projects to a subset of the curve graph of diameter $2$.}

\textit{Proof of Corollary \ref{CGDiam}}: Suppose $\Omega$ contains a pair of curves $\alpha, \beta$ that are distance $3$ or more in the curve graph. Then $\alpha$ and $\beta$ fill $S$, and $\Omega$ is contained in the intersection $\mathcal{L}_{k}(\alpha) \cap \mathcal{L}_{k}(\beta)$. Therefore by Theorem \ref{KLink}, $|\Omega|$ is bounded above by a sub-exponential function of $k$. However, in \cite{Aoug1} we constructed $k$-systems with cardinality bounded below by an exponential function of $k$, and therefore for $k$ sufficiently large, $|\Omega| < N_{k}(S).$ $\Box$

\subsection{Unit-Square tiled surfaces.} In this section, we assume $S$ is a closed surface. Another application of Theorem \ref{KLink} is to bound the number of short simple closed curves on square-tiled surfaces homeomorphic to $S$. Explicitly, a \textit{unit-square tiled surface} is a flat surface $\mathcal{S}$ obtained by gluing together finitely many copies of the unit square in $\mathbb{C}$, such that:
\begin{enumerate}
\item vertical edges glue to vertical edges, and similarly horizontal edges glue to horizontal edges;
\item Each vertex is adjacent to at least $4$ squares after the gluing;
\item the resulting surface is orientable. 
\end{enumerate} 

In particular, we do not require that a left-hand vertical edge glues to a right-hand vertical edge, or that a top horizontal edge glues to a bottom one. Each unit-square tiled surface is conformally equivalent to an area $1$ singular flat surface- a unit area surface with a flat metric away from finitely many singularities- obtained by scaling down the area of each square to one with area equal to the reciprocal of the total number of squares in the tiling. We also note that the set of all unit-square tiled surfaces homeomorphic to $S$ projects to a dense subset of $\mathcal{M}(S)$, the moduli space of complete finite volume hyperbolic metrics on $S$, where a square-tiled surface maps to the unique hyperbolic structure in its conformal class.

 In the special case that top edges glue to bottom edges and left edges glue to right ones, the resulting surface is a branched cover of the torus and is conformally equivalent to a so-called \textit{origami}, a type of translation surface whose $SL(2,\mathbb{R})$ orbit has important dynamical and algebro-geometric properties (\cite{HubLe}, \cite{Smit}). 

 If $\mathcal{S}$ is a unit-square tiled surface, let $N_{\mathcal{S}}(L)$ denote the number of homotopy classes of simple closed curves admitting a representative on $\mathcal{S}$ with length at most $L$. Finally, let $\mathcal{X}(S)$ denote the set of all unit-square tiled surfaces whose underlying topology is that of $S$. 

\textbf{Corollary \ref{Sqtile}.} \textit{For $S$ a closed surface, there exists a function $P_{S}$ which grows at most quasi-polynomially, such that }
\[ \sup\left\{N_{\mathcal{S}}(L): \mathcal{S} \in \mathcal{X}(S) \right\} \leq P_{S}(L).\]

\begin{figure}[H]
\centering
	\includegraphics[width=1.9 in]{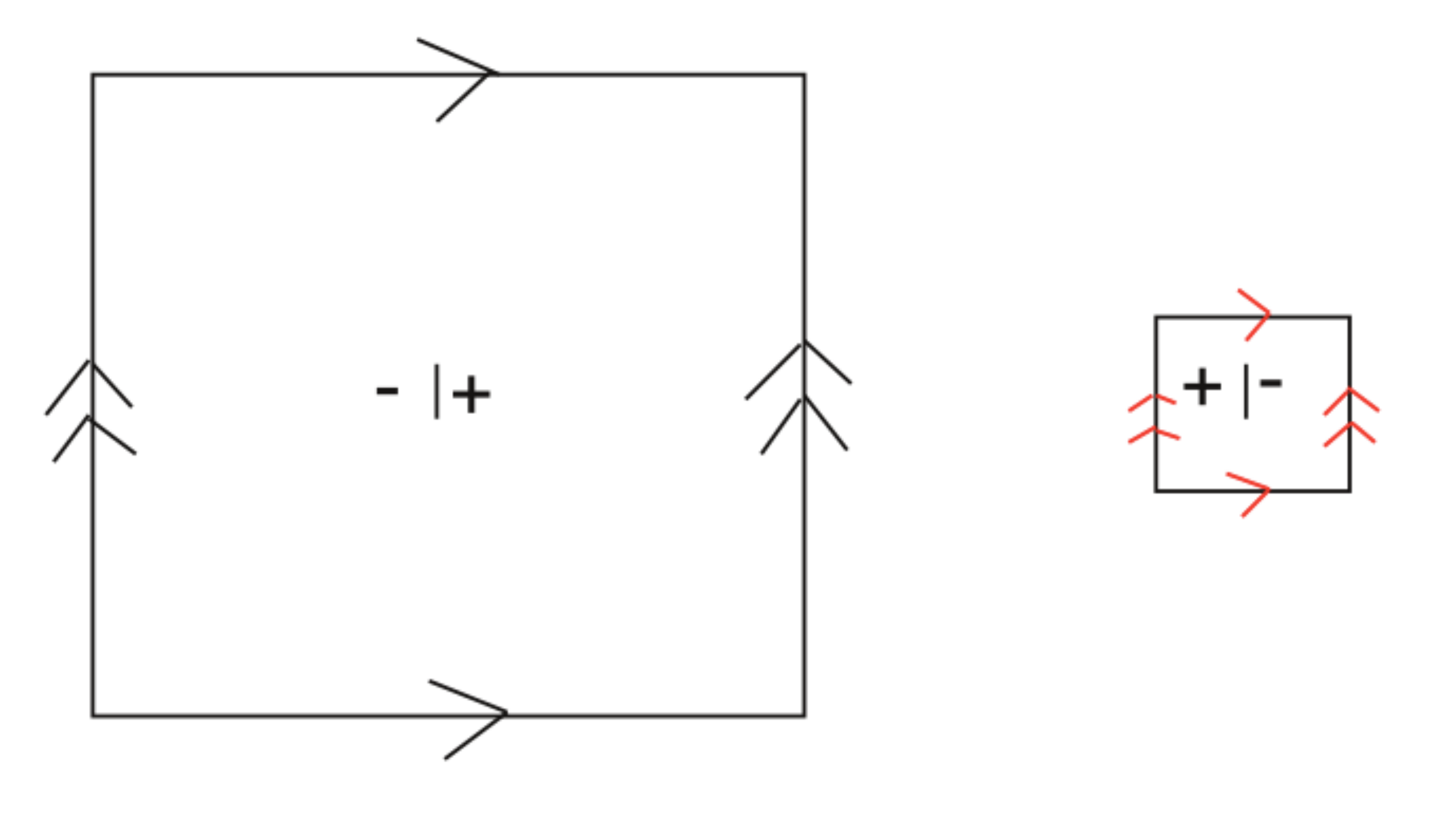}
\caption{A genus $2$ singular flat surface, consisting of two flat tori glued together via a small slit on the interior of each torus. The left hand side of the left slit is glued to the right hand side of the right slit. The area of the small torus is $\epsilon$ and the area of the larger is $1-\epsilon$. For $\epsilon$ very small, there will be many homotopy classes of simple closed curves with representatives on the smaller torus, all with short lengths.}
\end{figure}

\begin{remark} \label{ExplainingTile}
That there exists a polynomial upper bound of degree $\dim(\mathcal{T}(S))$ for $N_{\mathcal{S}}(L)$, for a \textit{fixed} unit-square tiled surface $\mathcal{S}$, follows from work of Rivin \cite{Riv} and also Mirzakhani \cite{Mirz}. However, this does not necessarily imply a uniform quasi-polynomial bound over all $\mathcal{X}(S)$. Indeed, given a fixed hyperbolic metric $\sigma$ on $S$, the number $N_{\sigma}(L)$ of simple closed geodesics of length $\leq L$, also satisfies a polynomial upper bound of degree $\dim(\mathcal{T}(S))$ (again by \cite{Riv} and \cite{Mirz}), but a \textit{uniform} upper bound for $N_{\sigma}(L)$, taken over all hyperbolic metrics $\sigma$, necessarily grows exponentially in $L$ (see for instance inequality $(3.14)$ of Proposition $3.6$ of \cite{Mirz}). 
There is also a polynomial upper bound for the number of (homotopy classes of) simple closed geodesics of length $\leq L$ for any fixed unit-area singular flat metric on $S$. However there can be no uniform bound (sub-exponential or otherwise) over all unit-area singular flat metrics on $S$, as the Figure 1 demonstrates.

\end{remark}

\textit{Proof of Corollary \ref{Sqtile}}: Given $\mathcal{S} \in \mathcal{X}(S)$, consider its \textit{vertical and horizontal curves} $v$ and $h$: $v$ (resp. $h$) is the multi-curve obtained by concatenating all vertical (resp. horizontal) midsegments of squares. The requirement that no vertical edges glue to horizontal edges guarantees that $v$ and $h$ are both multi-curves (potentially with many parallel components) and are distinct from each other. 

If $\alpha$ is a minimum length representative of a simple closed homotopy class on $\mathcal{S}$, then at the cost of increasing length by a factor of at most $\sqrt{2}$, we can homotope $\alpha$ so that it lies on the $1$-skeleton of $\mathcal{S}$. Therefore the length of $\alpha$ is roughly equal to the number of times it intersects both $v$ and $h$. We can assume that any two components of $v$ (resp. $h$) are not homotopic to each other, since deleting parallel components only reduces intersection number with other curves. Hence $N_{\mathcal{S}}(L)$ is bounded above by 
\[ |\mathcal{L}_{\lceil \sqrt{2}L \rceil}(v) \cap \mathcal{L}_{\lceil \sqrt{2} L \rceil}(h) |, \]
and thus the corollary follows by applying Theorem \ref{KLink}. $\Box$.

\subsection{Methods.}  We prove Theorems \ref{KSYS} and \ref{KLink} by appealing to the geometry of the mapping class group $\mbox{Mod}(S)$ and of Teichm{\"u}ller space $\mathcal{T}(S)$. 

\textbf{Idea for proof of Theorem \ref{KSYS} }: Given a $k$-system $\Gamma$, we construct an injective map from $\Gamma$ into a Cayley graph for $\mbox{Mod}(S)$ (with respect to a fixed finite generating set), such that the image of $\Gamma$ is contained in a ball of radius roughly $k$. Since $\mbox{Mod}(S)$ has exponential growth, the bound on $|\Gamma|$ follows. To accomplish this, we use formulas for measuring distance in $\mbox{Mod}(S)$, and also for relating intersection number to ``subsurface projections'', due to Masur-Minsky (\cite{MasMin2}) and Choi-Rafi (\cite{RafChoi}) respectively. 

\textbf{Idea for proof of Theorem \ref{KLink} }: Given a filling (multi-curve) pair $\alpha, \beta$, the objective is to bound the number of simple closed curves intersecting both $\alpha, \beta$ at most $k$ times. To do this, we use an estimate due to Choi-Rafi (\cite{RafChoi}) for measuring distance in the thick part of $\mathcal{T}(S)$ to find a thick hyperbolic surface $\sigma$ on which both $\alpha$ and $\beta$ have short representatives. It follows that the length of a curve on $\sigma$ is comparable to the number of times it intersects both $\alpha$ and $\beta$, and thus we complete the proof of Theorem \ref{KLink} by appealing to estimates of Rivin for the number of short simple closed geodesics on a thick hyperbolic surface (\cite{Riv}). 

\subsection{Notation and terminology.}  Given two quantities (or functions) $f, g$, by $f \asymp_{C} g$ we mean 
\begin{equation} \label{Coarse}
 \frac{1}{C}f - C \leq g  \leq C \cdot f + C.
\end{equation}
Generally the quantities $f,g$ will be functions of $k$ and $S$, or over (pairs of) simple closed curves, such as intersection number, or perhaps distance functions over $\mbox{Mod}(S)$ or $\mathcal{T}(S)$; moreover the constant $C$ will depend only on the topology of the underlying surface $S$. When the explicit constant $C$ is not of interest, we will suppress it by using the notation $f \asymp g$, meaning there exists some constant $C$ such that $f \asymp_{C} g$. 

By $f \prec g$ (respectively $f \succ g)$, we mean that there exists a constant $C$ such that the right-hand (resp. left-hand) inequality of (\ref{Coarse}) holds. We say $f$ is \textit{coarsely less than} or \textit{coarsely at most} $g$ to mean $f \prec g$, and that $f$ and $g$ are \textit{coarsely equal} if $f \asymp g$. 

By $f \prec^{+} g$, or $f \asymp^{+} g$, we mean that there exists some $C$ so that 

\[ f \leq g+ C, \]
or 
\[ f- C \leq g \leq f+ C, \]
respectively. 

Finally, all logarithms are base $2$. 

\subsection{Acknowledgements.} The author thanks Jayadev Athreya, Jeffrey Brock, Yair Minsky, Priyam Patel, Jenya Sapir, and Samuel Taylor for numerous insightful conversations. A weaker version of Theorem \ref{KSYS} appeared in the author's Ph.D thesis, and thus the author also thanks Andrew Casson and Dan Margalit for reading through the proof of this version of the theorem and for many helpful comments. The author was partially supported by NSF grants DMS 1005973, 1311844, and by NSF postdoctoral fellowship grant DMS 1502623.

\section{preliminaries}

\subsection{Curves and arcs.} A \textit{simple closed curve} on a surface $S$ is the image of an embedding $\phi :S^{1} \rightarrow S$. A curve is \textit{essential} if it is not homotopically trivial, and not homotopic into a neighborhood of a puncture, or parallel to a boundary component. A \textit{simple arc} on $S$ is either an embedding $\psi: (0,1) \rightarrow S$ with $\lim_{t\rightarrow 0} \psi(t), \lim_{t\rightarrow 1}\psi(t)$ coinciding with punctures, or an embedding of $[0,1]$ with endpoints lying on boundary components. A simple arc is \textit{essential} if it can not be homotoped to lie within a neighborhood of a puncture or boundary component. Homotopy between arcs is not required to fix boundary components point-wise.

A multi-curve (or multi-arc) is a disjoint union of simple closed curves (respectively arcs). 

 Given two homotopy classes of curves or arcs $\alpha, \beta$, their \textit{geometric intersection number} denoted $i(\alpha, \beta)$, is defined as 

\[ i(\alpha, \beta)= \min_{x \sim \alpha, y \sim \beta} |x \cap y|, \]

where $\sim$ denotes homotopy. If curves $\alpha, \beta$ achieve the geometric intersection number associated to the corresponding pair of homotopy classes, we say they are in \textit{minimal position}. A pair of simple closed curves $\alpha, \beta$ are in minimal position if and only if no connected component of $S \setminus (\alpha \cup \beta)$ is a bigon, which is a simply connected region bounded by one arc of $\alpha$ and one of $\beta$ (see section $1.2.4$ of \cite{Far-Mar}). 

A \textit{curve system} is a collection of pairwise non-homotopic, essential simple closed curves on $S$. If $\Lambda, \Gamma$ are two curve systems, we say $\Lambda$ is homotopic to $\Gamma$ as curve collections if there is a bijection from $\Lambda$ to $\Gamma$ such that the image of each curve in $\Lambda$ is homotopic to it. Then we define the geometric intersection number $i(\Gamma, \Gamma')$ by 

\[ i(\Gamma, \Gamma')= \sum_{\gamma \in \Gamma, \gamma' \in \Gamma'}i(\gamma, \gamma'). \]

A collection $\left\{\alpha_{1},..., \alpha_{n} \right\}$ of curves is said to \textit{fill} $S$ if $S \setminus \bigcup_{i} \alpha_{i}$ is a disjoint union of topological disks, once-punctured disks and boundary parallel annuli. 

\subsection{The mapping class group and the curve graph} The \textit{mapping class group} of $S$, denoted $\mbox{Mod}(S)$, is the group of orientation preserving homeomorphisms of $S$ fixing the boundary point-wise, up to isotopy. The \textit{extended mapping class group}, denoted $\mbox{Mod}^{\pm}(S)$, is the group of all isotopy classes of homeomorphisms (orientation preserving or reversing) of $S$ fixing the boundary point-wise. 

If $S$ is not a twice punctured disk, a torus with either $0$ or $1$ punctures, or a $4$-punctured sphere (or any of these surfaces with boundary components replacing some of the punctures), then the \textit{curve graph} of $S$, denoted $\mathcal{C}(S)$, is the graph whose vertices correspond to isotopy classes of essential simple closed curves on $S$, and two vertices span an edge exactly when the corresponding isotopy classes can be realized disjointly on $S$. Note that the curve graph of the $3$-holed sphere is the empty graph. If $S$ is a torus with $0$ or $1$ punctures, then adjacency in $\mathcal{C}(S)$ corresponds to simple closed curves intersecting once; if $S$ is the $4$-holed sphere, then adjacency corresponds to curves intersecting twice. 

Finally, if $S$ is the annulus, identify $S$ with the quotient of the hyperbolic plane by the action of an infinite cyclic subgroup of $PSL(2, \mathbb{R})$ generated by a hyperbolic matrix. Thus $S$ is homeomorphic to an open cylinder, and it admits a compactification $\bar{S}$ so that $S$ is homeomorphic to the interior of $\bar{S}$. Then vertices of $\mathcal{C}(S)$ correspond to geodesic simple arcs on $\bar{S}$ running from one boundary component to the other, and adjacency corresponds to disjointness. The curve graph is made into a metric space by identifying each edge with $[0,1]$. Let $d_{S}( , )$ denote distance in $\mathcal{C}(S)$. The curve graph admits an isometric (but not properly discontinuous) action of $\mbox{Mod}^{\pm}(S)$. 

Define $\mathcal{AC}(S)$, the \textit{arc and curve graph of $S$} to be the graph whose vertices correspond to isotopy classes of essential simple closed curves and arcs on $S$. As with $\mathcal{C}(S)$, two vertices are connected by an edge if and only if the corresponding isotopy classes can be realized disjointly. 

By a simple surgery argument, distance in $\mathcal{C}(S)$ is bounded above by a logarithmic function of intersection number (\cite{Hemp}, \cite{Lik}); given any two simple closed curves $\alpha, \beta$, 
\begin{equation} \label{Hempel}
d_{S}(\alpha, \beta) \leq 2 \log(i(\alpha, \beta)) +2.
\end{equation}

\subsection{Subsurface projections} A \textit{non-annular subsurface} $Y$ of $S$ is the closure of a complementary component of an essential multi-curve on $S$; an \textit{annular subsurface} $Y\subseteq S$ is a closed neighborhood of an essential simple closed curve on $S$, homeomorphic to $[0,1] \times S^{1}$. A subsurface is \textit{essential} if its boundary components are all essential curves, and it is not homeomorphic to a sphere with the sum of boundary components and punctures at most $3$. 

 Let $Y \subseteq S$ be an essential subsurface of $S$. Then there is a covering space $S^{Y}$ associated to the inclusion $\pi_{1}(Y)<\pi_{1}(S)$. While $S^{Y}$ is not compact, note that the Gromov compactification of $S^{Y}$ is homeomorphic to $Y$, and via this homeomorphism we identify $\mathcal{AC}(Y)$ with $\mathcal{AC}\left(S^{Y}\right)$. Then, given $\alpha \in \mathcal{AC}^{0}(S)$, we obtain a map $\pi_{Y}:\mathcal{AC}(S)\rightarrow \mathcal{AC}(Y)$ defined by setting $\pi_{Y}(\alpha)$ equal to its preimage under the covering map $S^{Y}\rightarrow S$. 

Technically, this defines a map from $\mathcal{AC}^{0}(S)$ into $2^{\mathcal{AC}^{0}(Y)}$ since their may be multiple connected components of the pre-image of a curve or arc, but the image of any point in the domain is a bounded subset of the range. Thus to make $\pi_{Y}$ a map we can simply choose some component of this pre-image for each point in the domain.

When $S$ is not an annulus, given an arc $a\in \mathcal{AC}(S)$, there is a closely related simple closed curve $\tau(a) \in \mathcal{C}(S)$, obtained from $a$ by surgering along the boundary components that $a$ meets. More concretely, let $\mathcal{N}(a)$ denote a thickening of the union of $a$ together with the (at most two) boundary components of $S$ that $a$ meets, and define $\tau(a)\in 2^{\mathcal{C}^{1}(S)}$ to be the essential components of $\partial(N(a))$. 

Thus we obtain a \textit{subsurface projection} map 
\[\psi_{Y}:= \tau \circ \pi_{Y}: \mathcal{C}(S)\rightarrow \mathcal{C}(Y)\]
 for $Y\subseteq S$ any essential subsurface. In practice, to obtain $\psi_{Y}(\alpha)$, consider the intersection of $\alpha$ with $Y$. If $\alpha$ is contained completely within $Y$, we define $\psi_{Y}(\alpha)= \alpha$; if $\alpha \cap Y = \emptyset$, then the projection $\psi_{Y}(\alpha)$ is undefined. Finally, if $\alpha \cap Y$ consists of a collection of arcs, define $\psi_{Y}(\alpha)$ to be the curves obtained by surgering those arcs via the process described in the previous paragraph. Note also that if $A$ is an annulus and $c$ is its core curve, the projection $\psi_{A}(c)$ is not defined. 

Then given $\alpha,\beta \in \mathcal{C}(S)$, define $d_{Y}(\alpha,\beta)$ by
\[ d_{Y}(\alpha,\beta):= \mbox{diam}_{\mathcal{C}(Y)}(\psi_{Y}(\alpha)\cup \psi_{Y}(\beta)).\]
We note that $\psi_{Y}$ is coarsely Lipschitz (see Lemma $2.3$ of \cite{MasMin2}):
\[ d_{Y}(\alpha, \beta) \leq 2d_{S}(\alpha, \beta)+ 2, \]
and furthermore, 
\begin{equation}\label{ProjInt}
 i(\alpha, \beta) \leq 2\cdot i(\psi_{Y}(\alpha), \psi_{Y}(\beta)) + 4.
\end{equation}

\subsection{The marking graph and the Masur-Minsky distance formula} Given $S$, let $P= \left\{\alpha_{1},...,\alpha_{n}\right\}$ be a complete subgraph of $\mathcal{C}(S)$. Then a \textit{marking} $\mu= \left\{\beta_{1},...,\beta_{n}\right\}$ is a certain decoration of $P$. That is, either $\beta_{i}= \alpha_{i}$, or $\beta_{i}= (\alpha_{i}, t_{i})$, where $t_{i}$ is a choice of \textit{transversal} for $\alpha_{i}$, meaning that $t_{i}$ is a diameter-$1$ set of vertices of the annular curve graph $\mathcal{C}(\alpha_{i})$. $P$ is called the \textit{base} of $\mu$. 

The marking $\mu$ is called \textit{complete} if $P$ is a pants decomposition, and every curve has a transversal, and $\mu$ is called \textit{clean} if each transversal $t_{i}$ is of the form $ \psi_{\alpha_{i}}(\gamma_{i})$, where 

\begin{enumerate}
\item $i(\gamma_{i}, \alpha_{i})=1$ and a regular neighborhood of $\gamma_{i} \cup \alpha_{i}$ is a torus with one boundary component; or
\item $i(\gamma_{i}, \alpha_{i})= 2$ and a regular neighborhood of $\gamma_{i} \cup \alpha_{i}$ is a sphere with $4$ boundary components; and in either case we require that 
\item $\gamma_{i}$ is disjoint from $\alpha_{j}$ for any $j \neq i$. 
\end{enumerate} 

The curve $\gamma_{i}$ is called a \textit{clean transverse curve} to $\alpha_{i}$. While $\gamma_{i}$ can only intersect one base curve, we note that clean transverse curves can intersect many other clean transverse curves. The \textit{marking graph} $M(S)$ of $S$ is the graph whose vertices are complete clean markings on $S$, and whose edges are of the following two forms:

\begin{enumerate}

\item twist: $(\mu, \mu')$ where bases of $\mu$ and of $\mu'$ agree, and for exactly one $i$, $t'_{i}$ is obtained from $t_{i}$ by Dehn twisting once about $\alpha_{i}$ (in the case that $\alpha_{i}$ and its clean transverse curve fill a torus) or by half Dehn twisting once about $\alpha_{i}$ (in the case that $\alpha_{i}$ and its clean transverse curve fill a $4$-holed sphere); 

\item flip: $(\mu, \mu')$ where $\mu'$ is obtained from $\mu$ by exchanging the roles of clean transverse curve and base curve for exactly one $i$: $\gamma_{i}$ becomes a curve of the base of $\mu'$ and $\alpha_{i}$ is its clean transverse curve. The new base is indeed a multi-curve because $\gamma_{i}$ does not intersect any of the other base curves. However, because clean transverse curves are allowed to intersect each other, the resulting complete marking may not be clean. Thus we replace the resulting marking with a clean marking that is \textit{compatible} with it, which is a clean marking that has the same base, and transversals are chosen to minimize distance with the original transversals in each annular curve graph. 

\end{enumerate}

There is a uniformly bounded number of choices for a clean marking that is compatible with a given complete marking. Moreover there is always at least one such choice, and the transversals of any cleaning are uniformly close, in the respective annular curve graphs, to those of the original complete marking (see Lemma $2.4$ of \cite{MasMin2}). Thus $M(S)$ is a locally finite, connected graph, admitting a properly discontinuous (but not free) isometric action of $\mbox{Mod}(S)$. The quotient of $M(S)$ by $\mbox{Mod}(S)$ is a finite graph, and therefore $M(S)$ is quasi-isometric to $\mbox{Mod}(S)$ with the word metric associated to any finite generating set. 

We can extend the subsurface projection operation to markings as follows. If $Y$ is an annulus whose core curve $\alpha$ is in the base of $\mu$, then we define $\psi_{Y}(\mu)$ to be the transversal to $\alpha$ in $\mu$ (if there is no transversal the projection is defined to be empty). Otherwise, $\psi_{Y}(\mu)$ is defined to be the usual subsurface projection of the base of $\mu$ to $Y$. 

The following formula, due to Masur and Minsky, allows for the computation of distance in $M(S)$ via subsurface projections (\cite{MasMin2}):

\begin{theorem} \label{SubSurface} There exists $D= D(S)$ such that for any $T>D$, the following holds. There exists $N$ such that for any $\mu_{1}, \mu_{2}$ complete clean markings, 

\[ d_{\mathcal{M}}(\mu_{1}, \mu_{2}) \asymp_{N} \sum_{Y \subseteq S} [[d_{Y}(\mu_{1}, \mu_{2})]]_{T}, \]

where $[[x]]_{T}= x$ for $x \geq T$ and $0$ otherwise.

\end{theorem}

Given complete clean markings $\mu = \left\{(\alpha_{1}, t_{1}),..., (\alpha_{n}, t_{n})\right\}$ and $\mu'= \left\{(\alpha'_{1}, t'_{1}), ..., (\alpha'_{n}, t'_{n}) \right\}$ we define their intersection number $i(\mu, \mu')$ to simply be the geometric intersection number

\[ i \left( \bigcup_{i} \alpha_{i} \cup \bigcup_{i} \gamma_{i}, \bigcup_{j} \alpha'_{j} \cup \bigcup_{j} \gamma'_{j} \right), \]
where $\gamma_{i}$ (resp. $\gamma'_{j}$) is the clean transverse curve corresponding to $t_{i}$ (resp. $t'_{j}$). 

It will be convenient to work with a modified version of $M(S)$, defined as follows. Given $k \in \mathbb{N}$, a \textit{k-marking} is a graph $\mu$ that fills $S$ (i.e., each complementary region of $\mu$ is simply connected or once-punctured) and such that $\mu$ has at most $k$ edges. Define 
\[j_{1}= \max_{\mu, \mu'}i(\mu, \mu'), \]
where the maximum is taken over all pairs of complete clean markings $\mu, \mu'$ connected by an edge in $M(S)$. Let $B$ be a bound on the number of edges in any complete clean marking $\mu$, interpreted as a graph on $S$. Finally, let 
\[ j_{2}= \max_{\Gamma}\min_{\mu}i(\Gamma, \mu), \]
were the maximum is taken over all graphs $\Gamma$ that fill $S$ with at most $B$ edges, and the minimum is taken over all complete clean markings $\mu$. 

Then for $j= \max(j_{1}, j_{2})$, consider the graph $M_{j,B}(S)$, whose vertices are $B$-markings, and whose edges correspond to pairs $(\Gamma, \Gamma')$ intersecting at most $j$ times. The graph $M_{j,B}(S)$ is connected by choice of $j,B$ and by the connectedness of $M(S)$, and the map $i:M(S)\rightarrow M_{j,B}(S)$ sending a complete clean marking to itself (interpreted as a $B$-marking) is a $\mbox{Mod}(S)$-equivariant quasi-isometry. We define the subsurface projection of a $B$-marking to be equal to the projection of its pre-image under the map $i$; it follows that Theorem \ref{SubSurface} applies as written to $M_{j,B}(S)$. 

\subsection{Hierarchy paths} As part of the proof of Theorem \ref{SubSurface}, given complete, clean markings $\mu, \mu'$ Masur and Minsky construct certain paths in $M(S)$, called \textit{hierarchy paths}, from $\mu$ to $\mu'$ (\cite{MasMin2}). A hierarchy path $H$ from $\mu$ to $\mu'$ can be identified with a collection $C_{H}$ of geodesics defined in curve graphs of various essential subsurfaces of $S$. Given a geodesic $h \in C_{H}$, let $D(h) \subseteq S$ denote the subsurface on which $h$ is defined. There exists a certain relation $\searrow^{d}$ (\textit{direct forward subordinacy}) on the geodesics in $C_{H}$ whose transitive closure generates a partial order on $C_{H}$ (denoted $\searrow$ and called \textit{forward subordinacy}); we will not need the details of this identification, nor the exact definition of forward subordinacy, and therefore we only record some important properties below:

\begin{enumerate}

\item Exactly one of the geodesics in $C_{H}$, called the \textit{main geodesic},  lives in $\mathcal{C}(S)$, and thus every other one is a geodesic in the curve graph of some proper subsurface. Furthermore, the length of the hierarchy path is equal to the sum of the lengths over all geodesics in $C_{H}$.

\item There exists $J= J(S)$ such that if $g$ is a geodesic in $C_{H}$ supported on some subsurface $Y \subseteq S$, then the length of $g$, which we denote by $|g|$, is within $J$ of $d_{Y}(\mu, \mu')$. Moreover, if $d_{Y}(\mu, \mu')>J$, then there exists a geodesic $h$ in the Hierarchy with $D(h)= Y$.

\item If $h \searrow^{d} g$, then $|\chi(D(h))|< |\chi(D(g))|$. Furthermore, given $g \in C_{H}$, the number of geodesics $h$ satisfying $h \searrow^{d} g$ is at most $|g|+4$. 

\item Let $m \in C_{H}$ denote the main geodesic. Then for any $h \in C_{H}, h \neq m$, we have $h \searrow m$.

\end{enumerate}

\subsection{Teichm{\"u}ller space and Rafi's formula} For this section, we assume $S$ has no boundary (but perhaps punctures). The \textit{Teichm\"{u}ller space} of $S$, denoted $\mathcal{T}(S)$, is the space of marked Riemann surfaces homeomorphic to $S$. Concretely,  $\mathcal{T}(S)$, as a set, is the collection of pairs $(\phi, \sigma)$ modulo a certain equivalence relation, where $\sigma$ is a finite area, complete hyperbolic surface homeomorphic to $S$ and $\phi: S \rightarrow \sigma$ is a homeomorphism. The equivalence relation is defined as follows: $(\phi, \sigma) \sim (\phi', \sigma')$ exactly when there exists an isometry $j: \sigma \rightarrow \sigma'$ such that $j \circ \phi$ is homotopic to $\phi'$. Given $x = (\phi, \sigma) \in \mathcal{T}(S)$, $\phi$ is called the \textit{marking}, or \textit{marking homeomorphism} of $x$.

We will assume that $\mathcal{T}(S)$ is equipped with the metric topology coming from the \textit{Teichm\"{u}ller metric}, denoted by $d_{\mbox{Teich}}(\cdot,\cdot)$. In this metric, the distance between two marked Riemann surfaces $x=(\phi_{1}, \sigma_{1})$ and $y=(\phi_{2}, \sigma_{2})$ is determined by the logarithm of the minimal dilatation associated to a quasiconformal map $\Phi: x \rightarrow y$ such that $\Phi \circ \phi_{1}$ is isotopic to $\phi_{2}$. $\mathcal{T}(S)$ is homeomorphic to $\mathbb{R}^{6g-6+2p}$, where $g$ is the genus of $S$ and $p$ is the number of punctures.

Fix $\epsilon>0$. The \textit{$\epsilon$-thick part} of $\mathcal{T}(S)$, denoted $\mathcal{T}_{\epsilon}$ is the set of all points in $\mathcal{T}(S)$ whose underlying hyperbolic metric has injectivity radius at least $\epsilon$; equivalently, it is the set of all marked hyperbolic surfaces on which every essential simple closed curve has length at least $2 \epsilon$. Let $x, y \in \mathcal{T}_{\epsilon}$, and let $\mu_{x}, \mu_{y}$ be the  shortest clean markings on $x,y$ respectively. Then the following formula due to Rafi relates the Teichm{\"u}ller distance $d_{\mathcal{T}}(x,y)$ to subsurface projections (\cite{Rafi}):

\begin{theorem} \label{Rafi} There exists $P>1$ such that 

\[ d_{\mathcal{T}}(x, y) \asymp \sum_{Y \subseteq S} [[d_{Y}(\mu_{x}, \mu_{y})]]_{P}+ \sum_{A \subset S}\log([[d_{A}(\mu_{x}, \mu_{y})]]_{P}), \]

where the first sum is over all non-annular essential subsurfaces $Y$, and the second is over all essential annuli. Moreover, we define $\log([[w]]_{P})$ to be equal to $0$ if $w<P$, and to be $\log(w)$ otherwise. 
\end{theorem}

We will also have use for the following coarse equality due to Choi-Rafi, which relates distance in the thick part of Teichm{\"u}ller space to the logarithm of intersection number (\cite{RafChoi}):
\begin{equation} \label{LogTeich}
\log(i(\mu_{x}, \mu_{y})) \asymp^{+} d_{\mathcal{T}}(x,y).
\end{equation}

\section{Bounds on $K$-systems}
In this section, we prove: 

\begin{theorem} \label{KSYS} \textit{Fix a surface $S$ with $\chi(S)<0$, and let $N_{S}(k)$ denote the clique number of $\mathcal{C}_{k}(S)$.  Then } 

\[ \log(N_{S}(k)) \prec k.  \]

\end{theorem}

\begin{proof} Fix $k>0$ and let $\Gamma= \left\{\gamma_{1},...,\gamma_{n}\right\}$ be a $k$-system on $S$. The strategy will be to produce a constant $W$ (depending only on $S$ and not on $k$), and a map $\Phi:\Gamma \rightarrow \mbox{Mod}(S)$ such that $(1)$ the pre-image of any point has cardinality at most $W$ and $(2)$ such that the image $\Phi(\Gamma)$ is contained in a ball of radius coarsely at most $k$ in $M(S)$. Since $M(S)$ has exponential growth, Theorem \ref{KSYS} follows.

To this end, we first show that distance in $M(S)$ is coarsely bounded above by intersection number:

\begin{lemma} \label{IntBound} . Given $\mu_{1}, \mu_{2} \in M(S)$, 

\[ d_{M}(\mu_{1}, \mu_{2}) \prec i(\mu_{1}, \mu_{2}). \]

\end{lemma} 

\begin{remark} We note that Lemma \ref{IntBound} is sharp: if $\mu_{1}$ is obtained from $\mu_{2}$ by applying $k$ twist moves about one base curve, then $d_{M}(\mu_{1}, \mu_{2})$ and $i(\mu_{1}, \mu_{2})$ will both be coarsely equal to $k$.  
\end{remark}

\begin{proof} To prove Lemma \ref{IntBound}, we use the following inequality, which follows from Rafi's distance formula and from (\ref{LogTeich}): there exists some constant $P=P(S)$ such that

\begin{equation} \label{RafiChoi}
\log(i(\mu_{1}, \mu_{2})) \succ \sum_{Y \subseteq S} [[d_{Y}(\mu_{1}, \mu_{2})]]_{P} +  \sum_{A \subset S} \log([[d_{A}(\mu_{1}, \mu_{2})]]_{P}), 
\end{equation}
where the first sum is over all non-annular essential subsurfaces of $Y$, and the second is taken over all essential annuli. We also recall the Masur-Minsky distance formula (\cite{MasMin2}), which asserts the existence of a constant $T=T(S)$ such that 

\begin{equation} \label{MasurMinsky}
 d_{M}(\mu_{1}, \mu_{2}) \asymp \sum_{Y \subseteq S}[[d_{Y}(\mu_{1}, \mu_{2})]]_{T}, 
\end{equation}
where the sum is taken over all essential subsurfaces of $S$, including annuli. Define $R_{1}, R_{2}$ by 

\begin{equation} \label{TwoSums}
 R_{1}:= \sum_{Y \subseteq S, \mbox{non-annular}} [[d_{Y}(\mu_{1}, \mu_{2})]]_{T}; R_{2}:= \sum_{A \subset S, A \hspace{1 mm} \mbox{annular}} [[d_{A}(\mu_{1}, \mu_{2})]]_{T}. 
\end{equation}
 Then the sum on the right-hand side of the Masur-Minsky distance formula (\ref{MasurMinsky}) is simply $R_{1}+R_{2}$. Define $R'_{2}$ to be the sum of the logarithm of all sufficiently large annular projections (larger than $P$); that is, $R'_{2}$ is simply the second sum on the right hand side of the Choi-Rafi formula (\ref{RafiChoi}).  

We first note that it suffices to assume that the threshold $P$ in (\ref{RafiChoi}) is equal to the threshold $T$ in (\ref{MasurMinsky}). Indeed, assume first that $P<T$. However we are free to raise the threshold $P$ until it equals $T$, as this will only make the right hand side of (\ref{RafiChoi}) smaller and therefore the inequality will remain true. Thus it suffices to assume that $T\leq P$. 

On the other hand, if $T<P$, we can also raise $T$ so that it coincides with $P$. This follows from the fact that the Masur-Minsky distance formula (\ref{MasurMinsky}) holds for all sufficiently large thresholds (however different thresholds require different coarse equality constants). Thus henceforth, we can assume that $P=T$. Hence in particular, (\ref{RafiChoi}) implies 

\begin{equation} \label{R1Bound}
\log(i(\mu_{1}, \mu_{2})) \succ R_{1}. 
\end{equation}

Thus, applying (\ref{MasurMinsky}), we have reduced Lemma \ref{IntBound} to proving 

\begin{equation} \label{R2Bound}
i(\mu_{1}, \mu_{2}) \succ R_{2}
\end{equation}

To this end, we first note that $R_{2}$ has coarsely at most $\log(i(\mu_{1}, \mu_{2}))$ summands. Indeed, by the third property of hierarchies listed in subsection $2.5$, any annulus $A$ for which $d_{A}(\mu_{1}, \mu_{2})$ is sufficiently large will appear in a hierarchy path from $\mu_{1}$ to $\mu_{2}$, and thus by properties $(3)$ and $(4)$ in subsection $2.5$ there are coarsely at most 
\[ R_{1}:= \sum_{Y \subseteq S} [[d_{Y}(\mu^{(j)}_{1}, \mu^{(j)}_{2})]]_{M} \]
such summands. Hence the desired bound follows from (\ref{R1Bound}).

We next claim that there exists a constant $L$ depending only on $S$, such that there exists at most $L$ essential annuli $A \subset S$ satisfying 

\begin{equation} \label{AnnBound}
d_{A}(\mu_{1}, \mu_{2})> i(\mu_{1}, \mu_{2})/\log(i(\mu_{1}, \mu_{2})).
\end{equation}

Assuming (\ref{AnnBound}), we have the bound 

\[ R_{2} \prec L \cdot i(\mu_{1}, \mu_{2}) + \log(i(\mu_{1}, \mu_{2})) \left[ \frac{i(\mu_{1}, \mu_{2})}{\log(i(\mu_{1}, \mu_{2}))} \right] \]
\[ = (L+1) i(\mu_{1}, \mu_{2}), \]
which implies the conclusion of Lemma \ref{IntBound}. Therefore it remains to prove the existence of such an $L$; assume by way of contradiction that no $L$ exists. Then there exists a sequence of pairs of markings 
\[(\mu^{(j)}_{1}, \mu^{(j)}_{2})_{i=1}^{\infty}\]
satisfying the property that if $\mathcal{S}^{(j)}$ denotes the number of annuli on to which the projection of the pair $\mu^{(j)}_{1}, \mu^{(j)}_{2}$ has distance at least 
\[i(\mu^{(j)}_{1}, \mu^{(j)}_{2})/\log(i(\mu^{(j)}_{1}, \mu^{(j)}_{2})), \]
then $\mathcal{S}^{(j)} \rightarrow \infty$. 

Henceforth, let $i^{(j)}$ denote the intersection $i(\mu^{(j)}_{1}, \mu^{(j)}_{2})$, and note that (\ref{RafiChoi}) implies that for each $j$, 

\begin{equation} \label{R'2Bound}
\log(i^{(j)}) \succ \sum_{A \subset S, \mbox{annular}} \log[[d_{A}(\mu^{(j)}_{1}, \mu^{(j)}_{2})]]_{T}=: R^{(j)}_{2},
\end{equation}
where the multiplicative and additive constants in the coarse equality depend only on $S$ and not on $j$. Exponentiating both sides of this inequality, we obtain 

\begin{equation} \label{ProductBound}
 i^{(j)} > 2^{-Q}\prod_{A \subset S}([[d_{A}(\mu^{(j)}_{1}, \mu^{(j)}_{2})]]_{T})^{1/Q}, 
\end{equation}
for some $Q>0$ depending only on $S$. However, since there are $\mathcal{S}^{(j)}$ summands of $R^{(j)}_{2}$ whose size is at least $i^{(j)}/\log(i^{(j)})$ (and since each summand is at least $1$), it follows that the product on the right hand side of (\ref{ProductBound}) is at least on the order of 

\[  2^{-Q}\left[i^{(j)}/\log(i^{(j)})\right]^{\mathcal{S}^{(j)}/Q},\]
which grows super-polynomially in $i^{(j)}$ because $\mathcal{S}^{(j)} \rightarrow \infty$, and this contradicts (\ref{ProductBound}). Hence $\mathcal{S}^{(j)}$ must be uniformly bounded. This completes the proof of Lemma \ref{IntBound}.

\end{proof}

Lemma \ref{IntBound} implies that Theorem \ref{KSYS} follows so long as we can construct the aforementioned map $\Phi:\Gamma \rightarrow \mbox{Mod}(S)$, so that $(1)$ the cardinality of each pre-image of $\Phi$ is bounded solely in terms of the topology of $S$, and $(2)$ there exists $\gamma\in \Gamma$ such that for any $\gamma' \in \Gamma$,

\[ i(\Phi(\gamma), \Phi(\gamma')) \prec k. \]

We will use the graph $M_{j,B}(S)$ as a model for $\mbox{Mod}(S)$; therefore, our goal is to associate a $B$-marking to each element of $\Gamma$. We first note that it suffices to assume that $\Gamma$ fills $S$. For if not, we can decompose $S$ into a disjoint union of subsurfaces $S_{1},S_{2},...$ such that each $\gamma$ is contained in one $S_{i}$ and such that each $S_{i}$ is filled by the subset $\Gamma_{i}$ of $\Gamma$ it contains. The number of such subsurfaces is bounded above solely in terms of the topology of $S$, and therefore the desired bound on $|\Gamma|$ follows by a bound on each $\Gamma_{i}$. Therefore, henceforth we assume that $\Gamma$ has one connected component. 

Given $\gamma \in \Gamma$, we will build a $B$-marking by starting with $\gamma$ and adding additional edges that are sub-arcs of other elements in $\Gamma$. To start, choose some $\gamma' \in \Gamma$ such that $i(\gamma, \gamma')\neq 0$. Then there exists a sub-arc $e$ of $\gamma'$ with endpoints on $\gamma$; then extend $\gamma$ to the graph $\gamma \cup e$. As we are assuming that elements of $\Gamma$ are in pairwise minimal position, no complementary component of $\gamma \cup e$ is a bigon. Now, we simply iterate; extend $\gamma \cup e$ to a larger graph by adding an edge $e'$ associated to a sub-arc of another element of $\Gamma$ intersecting $\gamma \cup e$. At each stage, $e'$ is chosen so that the absolute value of the Euler characteristic of the subsurface filled by the extended graph grows monotonically. 

Thus, after at most $|\chi(S)|$ iterations, we obtain a graph $\Phi(\gamma)$, built from $\gamma$ and from arcs of elements in $\Gamma$, that fills $S$. As there are complete clean markings with more than $|\chi(S)|$ edges, $\Phi(\gamma)$ is a $B$-marking. Furthermore, given $\gamma, \gamma' \in \Gamma$ it follows that 
\[ i(\Phi(\gamma), \Phi(\gamma')) \prec k, \]
since each edge of both graphs is a sub-arc of some element of $\Gamma$, and therefore any edge of $\Phi(\gamma)$ can intersect an edge of $\Phi(\gamma')$ at most $k$ times; thus the bound follows from the fact that both graphs are $B$-markings and have at most $B$ edges each by definition. Finally, we note that for each $\gamma \in \Gamma$, $\gamma$ is an embedded cycle in the graph $\Phi(\gamma)$; that is, $\Phi(\gamma)$ contains a graph-path homotopic to $\gamma$ which does not traverse any edge more than once. The number of such cycles is bounded above solely in terms of the number of edges of $\Phi(\gamma)$, and therefore the cardinality of any pre-image $\Phi^{-1}(\gamma)$ is bounded above solely in terms of $B$. This completes the proof of Theorem \ref{KSYS}.

\end{proof}

\section{Intersections of links}

In this section, we prove uniform bounds on the size of the intersection of $k$-links for a pair of filling curves $\alpha, \beta$:

\begin{theorem} \label{KLink}  There exists a function $r_{S}(k)$ depending only on the topology of $S$, which grows at most quasi-polynomially and which satisfies the following. Let $\alpha, \beta$ be simple closed curves on $S$ which fill $S$, and let $\mathcal{L}_{k}(\alpha)$ denote the set of all vertices in $\mathcal{C}_{k}(S)$ that are distance $1$ from $\alpha$. Then $|\mathcal{L}_{k}(\alpha) \cap \mathcal{L}_{k}(\beta)|\leq r_{S}(k)$. 
\end{theorem}

\begin{proof} The strategy of the proof is as follows: first we reduce to the case where the intersection number $i(\alpha, \beta)$ is bounded above by a quasi-polynomial function of $k$. To do this, we use the technology of hierarchies to argue that if $i(\alpha, \beta)$ is very large, there must exist some subsurface $S'$ of $S$ on which a definite number of intersections between $\alpha$ and $\beta$ accumulate. It will then follow that no curve in $\mathcal{L}_{k}(\alpha) \cap \mathcal{L}_{k}(\beta)$ can intersect $S'$, and thus we may restrict attention to its complement. Then using Choi-Rafi's estimate for distance in the thick part of Teichmuller space in terms of intersection number (\cite{RafChoi}), we find a hyperbolic surface $\sigma(\alpha, \beta)$ on which both $\alpha, \beta$ have bounded length (quasi-polynomial in $k$) representatives. It follows that the number of curves in $\mathcal{L}_{k}(\alpha) \cap \mathcal{L}_{k}(\beta)$ is comparable to the number of geodesics on $\sigma$ with length bounded above by some explicit quasi-polynomial function of $k$. Finally we appeal to a result of Rivin (\cite{Riv}) estimating the number of such bounded length curves on $\sigma$.

\subsection{Step 1: bound $i(\alpha,\beta)$}
Suppose there exists a non-annular essential subsurface $Y \subseteq S$ such that 
\[ d_{Y}(\alpha, \beta) > 4\log(k) + 10.\]
Then if $\gamma \in \mathcal{L}_{k}(\alpha) \cap \mathcal{L}_{k}(\beta)$, $\gamma$ must be disjoint from $Y$, or homotopic into the boundary of $Y$. Indeed, if $\gamma$ projected to $Y$, then $\psi_{Y}(\gamma)$ would intersect both $\psi_{Y}(\alpha), \psi_{Y}(\beta)$ at most $2k+4$ times, and thus by inequality (\ref{Hempel}), 
\[ d_{Y}(\gamma, \alpha), d_{Y}(\gamma, \beta) \leq 2\log(2k+4)+2, \]
and we obtain a contradiction by applying the triangle inequality in $\mathcal{C}(Y)$. Similarly, if there exists an annulus $A$ with $d_{A}(\alpha, \beta)> 2k+4$, no element of the intersection $\mathcal{L}_{k}(\alpha) \cap \mathcal{L}_{k}(\beta)$ can cross $A$. 

If such a non-annular subsurface $Y$, or an annulus $A$ exists, then consider its complement $S \setminus Y$, or $S \setminus A$. In either case, denote this complement by $S'$. Then any curve in $\mathcal{L}_{k}(\alpha) \cap \mathcal{L}_{k}(\beta)$ must be homotopic into $S'$. Furthermore, by construction of the projection map $\psi_{S'}$ and by inequality (\ref{ProjInt}), it follows that $\gamma \in \mathcal{L}_{k}(\alpha) \cap \mathcal{L}_{k}(\beta)$ only if it is in the intersection 
\[ \mathcal{L}_{2k+4}(\psi_{S'}(\alpha)) \cap \mathcal{L}_{2k+4}(\psi_{S'}(\beta)),\]
where these $(2k+4)$-links are in the graph $\mathcal{C}_{2k+4}(S')$. 

Thus, we claim that it suffices to assume that there does not exist a non-annular subsurface $Y$ such that
\[ d_{Y}(\alpha, \beta) > 4 \log(k) + 10, \]
or an annular subsurface $A$ such that 
\[ d_{A}(\alpha, \beta) > 2k+4. \]
Indeed, we can apply the above argument iteratively to remove any such subsurfaces. Each time we apply this argument, we concentrate on the intersection of links $\mathcal{L}_{k'}(\psi_{S'}(\alpha)) \cap \mathcal{L}_{k'}(\psi_{S'}(\beta))$ in a smaller subsurface, where, due to inequality (\ref{ProjInt}), $k'$ is coarsely at most $k$. Since the absolute value of the Euler characteristic must decrease at each stage, the argument eventually terminates, and we obtain a (possibly disconnected, possibly empty) subsurface $S'$ such that 
\begin{enumerate}
\item $d_{Y}(\psi_{S'}(\alpha), \psi_{S'}(\beta))\leq 4\log(k)+10$ for any essential non-annular subsurface $Y \subseteq S'$, and $d_{A}(\psi_{S'}(\alpha), \psi_{S'}(\beta))\leq 2k+4$ for any essential annular subsurface of $S'$;
\item  all elements of the intersection $\mathcal{L}_{k}(\alpha) \cap \mathcal{L}_{k}(\beta)$ must reside in $S'$;
\item  each of these curves must intersect the projections $\psi_{S'}(\alpha), \psi_{S'}(\beta)$ coarsely at most $k$ times; and 
\item $i(\psi_{S'}(\alpha), \psi_{S'}(\beta)) \prec i(\alpha, \beta).$
\end{enumerate}

Henceforth, we replace $S$ with $S'$, and $\alpha$ (resp. $\beta$) with its projection $\psi_{S'}(\alpha)$ (resp. $\psi_{S'}(\beta)$). Then using assumption $(1)$ above, we will bound the intersection number $i(\alpha, \beta)$ and this bound will later be used to bound $\mathcal{L}_{k}(\alpha) \cap \mathcal{L}_{k}(\beta)$, which by assumptions $(2), (3)$ and $(4)$, will imply the desired bound for the original surface $S$ and filling pair $\alpha, \beta$. 

Next, we will use the Choi-Rafi formula to bound $i(\alpha, \beta)$.  First, we augment $\alpha, \beta$ to markings $\mu_{\alpha}, \mu_{\beta}$ such that 
\[ i(\mu_{\alpha}, \mu_{\beta}) \asymp i(\alpha, \beta). \]
The marking $\mu_{\alpha}$ is obtained by projecting $\beta$ to the complement $S \setminus \alpha$. Concretely, $\mu_{\alpha}$ is the $B$-marking obtained by taking the union of $\alpha$ with a maximal collection of pairwise non-homotopic arcs of $\beta \cap (S \setminus \alpha)$, and similarly for $\mu_{\beta}$. Then the Choi-Rafi formula relating intersection number to subsurface projections states that
\begin{equation} \label{RafiChoi2}
 \log(i(\mu_{\alpha}, \mu_{\beta})) \prec \sum_{Y \subseteq S} [[d_{Y}(\mu_{\alpha}, \mu_{\beta})]]_{P} +  \sum_{A \subset S} \log([[d_{A}(\mu_{\alpha}, \mu_{\beta})]]_{P}). 
\end{equation}

Then the assumption that there are no large projections, together with the properties of hierarchy paths recorded in section $2$, will allow us to bound the right hand side from above in terms of $k$. Indeed, by property $(2)$ of hierarchy paths, every subsurface $Y$ with a large projection must appear in a hierarchy path from $\mu_{\alpha}$ to $\mu_{\beta}$, and furthermore, the length of the geodesic on that subsurface is coarsely the projection $d_{Y}(\mu_{\alpha}, \mu_{\beta})$. Thus by properties $(3)$ and $(4)$ of hierarchy paths, the right hand side of (\ref{RafiChoi2}) is coarsely at most 
\[ [\log(k)]^{f(S)}, \]
where 
\[ f(S) \asymp |\chi(S)|. \] 
To see this, define a \textit{weighted tree} $T$ to be a tree whose vertices are labeled by natural numbers. The \textit{size} of a weighted tree is the sum, over each vertex, of the weight of that vertex.   Then we claim that the right hand side of (\ref{RafiChoi2}) can be bounded above by the number of vertices in the following weighted tree $T$: $T$ has diameter $f(S)$ and base vertex $v$; each vertex of $T$ that is not adjacent to a univalent vertex has valence $\log(k)$ and weight $\log(k)$, and each univalent vertex of $T$ has weight $\log(k)$ as well. The base vertex $v$ represents the entire surface $S$; since the distance in the full curve graph $\mathcal{C}(S)$ between $\mu_{\alpha}$ and $\mu_{\beta}$ is coarsely at most $\log(k)$, we assign $v$ a weight of $\log(k)$. Since there are coarsely at most $\log(k)$ subsurfaces whose Euler characteristic has absolute value one less than the full surface $S$, $v$ has valence $\log(k)$. 

We iterate this argument for every proper subsurface, until we arrive at the annuli. Each annulus sees a projection of at most $k$, and therefore each contributes at most $\log(k)$ to the right hand side of (\ref{RafiChoi2}). These vertices are univalent as annuli do not contain any further proper subsurfaces. Thus the size of the weighted tree, and also the right hand side of (\ref{RafiChoi2}) is coarsely at most $[\log(k)]^{f(S)}$, as desired.

Then by exponentiating both sides of (\ref{RafiChoi2}), we obtain a quasi-polynomial upper bound on $i(\mu_{\alpha}, \mu_{\beta})$ and hence also on $i(\alpha, \beta)$. Let $w(k)$ denote this bound. 

\subsection{Step 2: Find the hyperbolic surface $\sigma(\alpha, \beta)$}

Let $S(\alpha)$ (resp. $S(\beta)$)  denote a hyperbolic surface in the thick part of Teichm{\"u}ller space minimizing the length of $\mu_{\alpha}$ (resp. $\mu_{\beta}$). Though $S$ may have both boundary components and punctures, we treat each boundary component as a puncture when hyperbolizing. Thus $\sigma(\alpha), \sigma(\beta)$ are both finite area complete hyperbolic surfaces of the same topological type, potentially disconnected and with parabolic cusps but without boundary. In the case that $S'$ is disconnected (and therefore $\sigma(\alpha), \sigma(\beta)$ are also disconnected), we define $\mathcal{T}(S')$ to be the product of the Teichm{\"u}ller spaces of its connected components, equipped with the sup metric. Then the Choi-Rafi estimate  (\ref{LogTeich}), together with the bound obtained in the previous subsection, implies 

\begin{equation} \label{RafiChoiAlBe}
\log(w(k)) \succ^{+} d_{\mathcal{T}}(\mu_{\alpha}, \mu_{\beta}).
\end{equation}

In particular, $\alpha$ admits a representative on $S(\beta)$ whose length is coarsely at most $w(k)$. This follows, for instance, by the fact that Thurston's Lipschitz metric on Teichm{\"u}ller space is bounded above by the Teichm{\"u}ller metric. 

Since the length of $\beta$ is bounded on $S(\beta)$ in terms of only the topology of the surface $S$, it follows that $\alpha \cup \beta$ has length bounded coarsely from above by $w(k)$ on $S(\beta)$. We set $\sigma(\alpha, \beta):= S(\beta)$.

\subsection{Step 3: bounding the number of short curves on $\sigma(\alpha, \beta)$.}
Thus the geodesic representatives for $\alpha$ and $\beta$ (which by abuse of notation we also refer to as $\alpha$ and $\beta$) decompose $\sigma(\alpha, \beta)$ into hyperbolic polygons $P_{1},...,P_{N}$ and possibly also a finite number of once-punctured regions with piecewise geodesic boundaries, such that:

\begin{enumerate}
\item for each $i =1,...,N$, each side of $P_{i}$ has hyperbolic length (coarsely) at most $w(k)$;
\item for each $i= 1,...,N$, $P_{i}$ has a uniformly bounded number of sides (in terms only of the topology of $S$). 
\end{enumerate}

Property $(1)$ follows from the previous subsection. Property $(2)$ follows from (\ref{EulChar}) below, and in particular it is true for any filling pair on $S$, regardless of intersection number. Indeed, let $N$ denote the number of simply-connected components of $S \setminus (\alpha \cup \beta)$. A basic Euler characteristic argument (together with the observation that $\alpha \cup \beta$ constitutes a $4$-valent graph) yields the equality 
\begin{equation} \label{EulChar}
\chi(S)= N- i(\alpha, \beta). 
\end{equation}
Since each of these regions has at least $4$ sides, it follows that no region can have more than $4|\chi(S)|+5$ sides. Note that the same bound applies for the number of sides of any of the once-punctured complementary regions. Indeed, if there exists a once-punctured region $R$ with at least $4$ sides, then $\alpha, \beta$ will be in minimal position on the surface $\tilde{S}$ obtained by filling in that puncture as this does not create any bigons. Thus $R$ can not have more than $4|\chi(S)|+5$ sides, by applying the same argument on $\tilde{S}$. 

Therefore, each $P_{i}$ has diameter coarsely bounded above by $w(k)$, and this implies that the length of an arc of any geodesic contained within one of the $P_{i}$'s has length coarsely at most $w(k)$. Moreover, although the diameter of a once-punctured region is infinite, the same argument implies that the length of an arc of a geodesic contained within one of the once-punctured complementary regions is also coarsely at most $w(k)$. Hence if $\gamma \in \mathcal{L}_{k}(\alpha) \cap \mathcal{L}_{k}(\beta)$, the geodesic representative for $\gamma$ on $\sigma(\alpha,\beta)$ has hyperbolic length at most $Z\cdot(k \cdot w(k))+ Z$, for some constant $Z>0$. That is, if $N_{\sigma}(Z\cdot(k \cdot w(k))+ Z)$ denotes the collection of simple closed geodesics on $\sigma(\alpha, \beta)$ of length at most $Z\cdot(k \cdot w(k))+ Z$, we have 

\[ \mathcal{L}_{k}(\alpha) \cap \mathcal{L}_{k}(\beta) \subseteq N_{\sigma}(Z\cdot(k \cdot w(k))+ Z). \]
Rivin (\cite{Riv}) has shown that there exists some constant $V= V(\sigma)$ such that for any $L$,
\[ N_{\sigma}(L) \leq V(L)^{|\dim(\mathcal{T}(S))|}+V. \]
The constant $V$ necessarily diverges as injectivity radius decays to $0$; however we have chosen $\sigma(\alpha, \beta)$ to be uniformly thick, independent of the choice of filling pair $\alpha, \beta$. This completes the proof of Theorem \ref{KLink}.

\end{proof}

As shown in the introduction, as a corollary we obtain the following bound on the diameter of a large $k$-system in the curve graph:

\begin{corollary} \label{CGDiam} Let $\Omega$ be a $k$-system on $S$ with $|\Omega|= N_{k}(S)$. Then for all sufficiently large $k$, $\Omega$ projects to a subset of the curve graph of diameter $2$.
\end{corollary}

Also demonstrated in the introduction was the uniform quasi-polynomial bound on the growth of short simple closed geodesics on unit-square tiled surfaces:

\begin{corollary} \label{Sqtile} For $S$ a closed surface, there exists a function $P_{S}$ which grows at most quasi-polynomially, such that 
\[ \sup\left\{N_{\mathcal{S}}(L): \mathcal{S} \in \mathcal{X}(S) \right\} \leq P_{S}(L).\]
\end{corollary}

\end{document}